%% file: ununs2.tex
\documentclass[11pt]{amsart}
\usepackage{amsmath}
\usepackage{amssymb}
\usepackage{amsfonts}
\usepackage{amsthm}
\usepackage{verbatim}
\usepackage{amscd}
\usepackage{cite}
\usepackage{leftidx}
\usepackage{enumerate}
\usepackage{txfonts}
\usepackage{manfnt}
\usepackage{amscd}
\usepackage[mathscr]{eucal}
\usepackage{hyperref}
\usepackage{mathpazo}

\input{myheader3.tex}

\newtheorem{thm}{Theorem} 

\newtheorem*{thm*}{Theorem}
\newtheorem*{prop*}{Proposition}
\newtheorem{cor}[thm]{Corollary}
\newtheorem*{cor*}{Corollary}

\newtheorem{lem}[thm]{Lemma}
\newtheorem{lem*}{Lemma}

\newtheorem*{claim*}{Claim}
\newtheorem{prop}[thm]{Proposition}

\newtheorem{defn}[thm]{Definition}
\theoremstyle{remark}

\newtheorem{rem}[thm]{Remark}
\newtheorem{crit-rem}[thm]{Critical remark}
\newtheorem{remarks}[thm]{Remarks}

\newtheorem{example}[thm]{Example}
\newtheorem*{example*}{Example}

\newtheorem*{defn*}{Definition}

\begin{document} 
\title{On the syzygies of reducible curves}
\author 
{Ziv Ran}



\date {\today}


\address {\nl UC Math Dept. \nl
Big Springs Road Surge Facility
\nl
Riverside CA 92521 US\nl }

\email {ziv.ran @ucr.edu}

\begin{abstract}

A theorem of Green says that a line bundle of degree $\geq 2g+1+p$ on a smooth
curve $X$ of genus $g$ has property $\mathrm N_p$. We prove a similar conclusion
for certain singular,  reducible curves $X$ under suitable
degree bounds over all irreducible components of $X$.

\end{abstract}

\maketitle

\section{Statements}
For a smooth curve $X$ of genus $g$, one of
the main results about  line bundles $L$ 
of 'large' degree $d$ is Mark L. Green's theorem (see \cite{green-koszul} or
\cite{laz-pos}, \S 1.8) that for all $0\leq p\leq d-2g-1$,  $L$ has property $\mathrm N_p$.
This means that $L$ is very ample and projectively normal ($p=0$), the homogeneous  ideal 
of $\phi_L(X)$
is generated by quadrics ($p=1$), and the $(p-1)$st order syzygies of
$\phi_L(X)$ are linear $(p\geq 2)$. This theorem is a natural
extension of the trivial facts that  $L$ is globally generated (resp. very ample) 
 if $d\geq 2g$ (resp. $d\geq 2g+1$).\par
As property $N_p$ has strong implications for the extrinsic geometry and the syzygies of the projective
image of $X$ via sections of $L$, one naturally seeks generalizations of these results to possibly singular curves.
The obvious generalizations to the case of nodal, possibly reducible, curves, involving
assumptions on the total degree,  trivially fail.
Indeed clearly, any numerical condition guaranteeing good behavior on reducible curves must involve
irreducible components or  subcurves
of $X$. One such condition called \emph{balancedness} was considered by Caporaso \cite{caporaso-imrn}.
A different condition was considered by Franciosi and Tenni \cite{franciosi-tenni} and used to prove
property $\mathrm N_0$.
In this paper we focus on a  condition involving all irreducible components,
and with it prove a version of Green's theorem for all $p\geq 0$. To state the result we need
some definitions.\par
A \emph{deeply lci} variety $X$ is a pure reduced variety such that every component sum $Z$ of $X$
is lci and meets the complementary component sum $Z^+:=\overline{X\setminus Z}$ in a Cartier divisor
on $Z$. Any locally planar curve is deeply lci (see Remark \ref{remarks}, (ii) below). 
Let $X$ be a deeply lci curve, $L$ a line bundle on $X$ and $W\subset H^0(L)$ a subspace
of dimension $m$. For a component sum $Z$ of $X$, we let $L_Z:=L\otimes\O_Z$, let $W_Z\subset H^0(L_Z)$
denote the  linear system on $Z$ induced by $W$, $m_Z$ its dimension,
and $g_{Z,X}$ the rank of the restriction map $H^0(\omega_X)\to H^0(\omega_X\otimes\O_Z)$.
\begin{thm}\label{N_p-thm}
Let $X$ be a deeply lci curve and $W\subset H^0(L)$ a linear system on $X$. 
Then $W$ has property $\mathrm N_p$ if for all irreducible components $Z\leq X$, 
\eqspl{}{
p\leq m_Z-g_{Z,X}-2. 
}
\end{thm} 
Let $g_Z$ denote the genus of $Z$, i.e. $h^0(\omega_Z)$, and $n_Z=Z.Z^+$ .
If $W=H^0(L)$, a crude estimate yields $m_Z\geq \deg_Z(L)-g_Z+1-n_Z$, hence
\begin{cor}
A line bundle $L$ on a deeply lci curve has property $\mathrm N_p$ if for all irreducible components $Z$,
\eqspl{}{
p\leq \deg_Z(L)-g_Z-g_{Z,X}-n_Z-1 
}
\end{cor}
If $X$ is nodal, then an elementary consequence of the residue theorem is
\eqspl{}{
g_{Z,X}=g_Z+n_Z-b_0(Z^+)
} where $b_0(Z^+)$ is the number of connected
components of $Z^+$. Hence
\begin{cor}
Notations as above, if $X$ is nodal, then $L$ has property $\mathrm N_p$ if for
all irreducible components $Z$,
\eqspl{}{
p\leq \deg_Z(L)-2g_Z-2n_Z-1+b_0(Z^+).
}
\end{cor}
Here the case $X$ irreducible is precisely (an essentially trivial generalization of) Green's original 
result. 
For $p=0$ this Corollary implies Caporaso's Theorem 2.2.1 in \cite{caporaso-imrn}.
We are not aware of any results in the literature about the reducible case for $p>0$.\par
Our proof of Theorem \ref{N_p-thm} follows Green's proof by using duality in Koszul
cohomology to reduce the statement to the following vanishing theorem for Koszul $\cK_{t, 0}$ cohomology:
\begin{thm}\label{vanishing-thm}
Notations as in Theorem \ref{N_p-thm}, the Koszul cohomology
\eqspl{}{
\cK_{t,0}(X, \omega_X,  L, W)=0, \forall t\geq \max\limits_Z(g_{Z,X}-m_Z+m) 
} the max being over all irreducible components $Z$.
\end{thm}
In fact, this result is a special case of a more general vanishing theorem
extending Green's theorem (3.a.1) in \cite{green-koszul}. To state this, let $E$ be a vector
bundle on $X$ and for a component sum $Z$, denote by $e_Z$ the rank of
the restriction map $H^0(E)\to H^0(E_Z)$.
\begin{thm}\label{vanishing-e-thm}
Notation as above, the Koszul cohomology 
\eqspl{}{
\cK_{t,0}(X, E,  L, W)=0}
provided
\eqspl{t-lower-bound-eq}{
t\geq \max\limits_Z(e_Z-m_Z+m) 
} the max being over all irreducible components $Z$ of $X$.
\end{thm}
\begin{example}
Suppose $X$ contains a component $Z$ of genus $g_Z>0$ with
complement $Y=Z^+$. Then, assuming $L_Y(-Y.Z)$ is nonspecial and $W=H^0(L)$,
we have for  \[W(-Z):=\ker (W\to H^0(L_Z)), \ \ 
t=\deg_Y(L)-n_Z-g_Y+1=\dim W(-Z),\]
that $\cK_{t,0}(X, \omega_X, L)\neq 0$ because it contains 
$\bigwedge\limits^tW(-Z)\otimes H^0(\omega_Z)$. Hence by duality
$\cK_{m-2-t, 2}(X, L)=\cK_{\deg_Z(L)-g_Z-1, 2}(X, L)\neq 0$, so $L$
fails to have property $\mathrm N_{\deg_Z(L)-g_Z-1}$.
For instance, the canonical system itself fails to have property $N_{g_Z+n_Z-3}$
provided $g_Z>0$.
\end{example}
\section{Basics }
Throughout this paper we work with complex projective reduced, pure, possibly reducible varieties,
usually curves.
\begin{defn}
A pure, reduced variety $X$ is said to be deeply lci (resp. deeply Gorenstein)  if 
\par (i) every component
sum $Z\leq X$ is lci (resp. Gorenstein);\par
(ii) for every component sum $Z$ with complement $Z^+:=\overline{X\setminus Z}$,
the schematic intersection $Z\cap Z^+$ is a Cartier divisor on $Z$.
\end{defn}
\begin{remarks}\label{remarks}
\begin{enumerate}

\item Normal crossings  $\implies$ deeply lci
$\implies$ deeply Gorenstein. 
\item Any variety that is locally a hypersurface, e.g. a locally planar curve,
 is deeply lci. Indeed any component sum $Z$ is also
locally a hypersurface, hence lci; and because $Z^+$ is 
locally a divisor on a smooth variety containing 
$X=Z\cup Z^+$,
it meets $Z$ in a Cartier divisor.

\item Not every lci (resp. Gorenstein) variety is deeply lci (resp. Gorenstein): for example,
$V(xy, zw)\subset\A^4$ has as component sum $V(x,z)\cup V(y,w)$ which is not even Cohen-Macaulay.
\item if $Z$ is a component sum on $X$ deeply lci, then
by basic properties of dualizing sheaves, we have an 'adjunction formula' 
\eqspl{}{
\omega_X\otimes\O_Z=\omega_Z(Z. Z^+).
} We believe this formula holds in the deeply Gorenstein case but 
have not checked it and will not need it.
\end{enumerate}
\end{remarks}

%
For a reduced Gorenstein curve $X$, the genus $g_X$ (sometimes denoted $g(X)$) is the arithmetic genus, i.e. the number such
that $\deg(\omega_X)=2g_X-2, h^0(\omega_X)=g_X$.
\begin{defn}
(i) A line bundle $L$ on $X$ is said to be $k$- numerically nonspecial or $k$-nuns if
\[\deg(L)\geq 2g_X-2+k.\]\par
(ii) $L$ is said to have expected dimension if $H^1(L)=0$.\par 
(iii) $L$ is said to be  $k$-uniformly numerically nonspecial or $k$-ununs if for all subcurves
 $Y\subset X$, $L_Y$ is $k$-nuns on $Y$.
 \par (iv)  $L$ is said to be $k$-spanned if for every ideal
 $I$ of colength $k$ on $X$, the natural map $H^0(L)\to H^0(L/IL)$ is surjective.
\qed
\end{defn}

Thus, 1-spanned is globally generated, 2-spanned is very ample, etc. It is well known that for $X$ irreducible nodal,
$k$-nuns implies $(k-1)$-spanned. Known results about these notions \footnote{
We thank F. Viviani for these references} include:\par
- Catanese and Franciosi \cite{catanese-franciosi} have proven (assuming only
that $X$ has planar singularities) that 1-ununs implies expected dimension, i.e.
$h^0(L)= \deg(L)-g(X)+1$.\par - Catanese, Franciosi, Hulek and Reid \cite{catanese-franciosi-hulek-reid}
have proven, again for more general curves $X$, 
that $k$-ununs implies $(k-1)$-spanned.\par 
- Caporaso \cite{caporaso-imrn} has some related results for balanced line bundles.\par
- As mentioned above, Franciosi and Tenni \cite{
franciosi-tenni} have proven that 3-ununs implies property $N_0$.\par
By comparison, our main result here is that for all $X$ nodal, $k$-ununs implies  $N_{k-3}$.\par
The following Lemma will be used.
\begin{lem}\label{j-lem}
Let $p$ be a node on $X$, let $J_{a,b}$ be the ideal of type $(x^a, y^b)$ of colength
$e=a+b-1$ cosupported at $p$,  let $\pi:X'\to X$ be the blowing-up of $p$, and
$L'$ the unique line bundle on $X'$ of degree $\deg(L)-a-b$ such that $\pi_*(L')=J_{a,b}L$.
If $L$ is $k$-ununs on $X$, then $L'$ is $(k-e)$-ununs on $X'$.
\end{lem}
\begin{proof} We may assume $0<b\leq a\leq e$.
Let $Y'$ be a subcurve of $X'$ and suppose to begin with that $Y'$ contains both preimages 
of $p$. Then $Y'$ is the blowing-up of $p$ on a uniquely determined subcurve
 $Y$ of $X$ containing $p$, and we have
\[\deg(L'_{Y'})=\deg(L_Y)-a-b\geq 2g(Y)-2+k-a-b=2g(Y')-2+k+1-e,\]
which is good enough. If $Y'$ contains precisely one preimage of $p$, it maps isomorphically to a 
subcurve  $Y$
of $X$ through $p$ and we have
\[\deg(L'_{Y'})\geq \deg(L_Y)-a\geq 2g(Y)-2+k-a\geq 2g(Y')-2+k-e\]
as $a\leq e$.  If $Y'$ contains no preimage of $p$ then $L|_{Y'}$ is trivially $k$-nuns. This concludes the proof. 
\end{proof}
As a warmup, we will prove the following known result
\begin{prop}\label{main-thm} Let $L$ be a line bundle on a nodal curve.\par
(i) If $L$ is 1-ununs , then $h^0(L)=\deg(L)-g+1$.\par
(ii) If $L$ is $k$-ununs for $k>1$, then $L$ is $(k-1)$-spanned.\par
\end{prop}
\begin{proof}
We begin by proving that if $L$ is 1-ununs, then 
\eqspl{vanishing}{h^0(\omega_X(-L))=0,} which by 
Riemann-Roch and Serre Duality implies \eqspl{good-sections}{h^0(L)=\deg(L)-g(X)+1.}
To prove \eqref{vanishing}, we use induction on the number of irreducible components of $X$. If $X$ is irreducible, the result is clear. For the induction step, let $s\in H^0(\omega_X(-L))$.
 Note that
\[\deg(L)=\sum\limits_Z \deg(L_Z)>\deg(\omega_X)=\sum\limits_Z\deg(\omega_X|_Z)\]
the sum being over all irreducible components $Z$ of $X$. Therefore there exists an irreducible component $Z$
such that \[\deg(L_Z)>\deg(\omega_X|_Z)\geq\deg(\omega_Z).\]
Therefore $s|_Z=0$. Let $Y=\overline{X\setminus Z}$ be the complementary curve.
Then since $s$ vanishes on $Y\cap Z$, it 
may be viewed as a section of $\omega_Y(-L_Y)$. Since $L_Y$ is 1-ununs, it follows that $s=0$.
This proves \eqref{vanishing}\par
Next, we will show that if $L$ is $k$-ununs, then
 for any ideal $J$ of colength $e\leq k-1$ on $X$, 
we have $h^0(JL)=\deg(L)-g(X)+1-e$ (this is equivalent to $L$ being $(k-1)$-spanned).
 If $J$ is invertible, it
 suffices to note that  the line bundle $JL$ is 1-ununs, then use \eqref{good-sections} for $JL$
 in place of $L$. 
Therefore we may assume $J$ is cosupported at the nodes. Using an obvious induction,  
we may assume $J$ is cosupported at a single node  $p\in X$. If $J$ is invertible,
again we may conclude by applying \eqref{good-sections} to $JL$.
If not, the results of \cite{hilb}  show that $J$ has the form $J_{a,b}$ as in Lemma  \ref{j-lem} 
and it suffices to apply \eqref{good-sections} 
to $L'$ on $X'$.\par
\end{proof}
\section{Proofs}
Green's duality theorem (\cite{green-koszul}, (2.c.6)) is a consequence of Serre duality and goes through
verbatim in the Gorenstein case. 
Now property $\mathrm N_p$ for $L$ is equivalent to the vanishing of $\cK_{p,2}(X, L)$,
which by duality is equivalent to the vanishing of $\cK_{m-2-p, 0}(X, \omega_X, L), m=h^0(L)$.
Therefore Theorem \ref{N_p-thm} is a consequence of Theorem \ref{vanishing-thm}
which in turn is a special case of Theorem \ref{vanishing-e-thm}, so it will suffice to prove the latter.\par
The proof is by induction on the number of irreducible components of $X$. When $X$ is irreducible,
Green's original proof of Theorem (3.a.1) in \cite{green-koszul} applies. 
In fact his argument proves the following more general statement, which we will need.
\begin{lem}[Green]\label{green-lem}
With notations as above, if  $X$ is irreducible and $V\subset H^0(E)$ is any subspace,
then the natural map
\eqspl{injectivity-v-eq}{
\bigwedge\limits^tW\otimes V\to \bigwedge\limits^{t-1}W\otimes H^0(E(L))
} is injective for all $t\geq \dim(V)$.\qed\end{lem}
For the induction step, pick
any component $Z$ and let $Y=Z^+$ be the complementary subcurve.
Let $V_Z\subset H^0(E_Z)$ be the image of $V\subset H^0(E)$ and similarly for $V'_Z\subset H^0(E(L)_Z)$
and for $V"_Z\subset H^0(E(-L)_Z)$. 
 Consider the horizontally exact, vertically semi-exact (= arrows compose to zero) diagram
\eqspl{}{
\begin{matrix}
0\to&\bigwedge\limits^{t+1}W\otimes H^0(E_Y(-Y.Z-L))&\to&\bigwedge\limits^{t+1}W\otimes H^0(E(-L))&\to&
\bigwedge\limits^{t+1}W\otimes V"_Z&\to 0\\
&\downarrow&&\downarrow&&\downarrow&\\
0\to&\bigwedge\limits^tW\otimes H^0(E_Y(-Y.Z))&\to&\bigwedge\limits^tW\otimes H^0(E)&\to&
\bigwedge\limits^tW\otimes V_Z&\to 0\\
&\downarrow&&\downarrow&&\downarrow&\\
0\to&\bigwedge\limits^{t-1}W\otimes H^0(E_Y(-Y. Z+L))&\to&\bigwedge\limits^{t-1}W\otimes H^0(E(L))&\to&
\bigwedge\limits^{t-1}W\otimes V'_Z&\to 0
\end{matrix}
}
Each column is a complex and the cohomology of the middle vertical column is 
$\cK_{t,0}(X, E, L, W)$, which we want to show
vanishes. \par
Let's next study the right column. I claim that the right bottom vertical map is injective,
a fortiori the right column is exact. Let $W(-Z)$ denote the kernel of the restriction map
$W\to W_Z$. Then $\bigwedge\limits^tW$ is filtered with quotients $\bigwedge\limits^iW(-Z)\otimes
\bigwedge\limits^{t-i}W_Z, i=0,...,m-m_Z$. This filtration induces one on the right column, whose
quotients have the form of a tensor product of a fixed vector space, viz. $\bigwedge\limits^iW(-Z)$, 
with a complex
\[\bigwedge\limits^{t+1-i}W_Z\otimes V"_Z\to\bigwedge^{t-i}W_Z\otimes V_Z\to\bigwedge\limits^{t-i-1}W_Z
\otimes V'_Z.\] 
We have
\[t-i\geq t-(m-m_Z)\geq e_Z.\]
By Green's Lemma \ref{green-lem} above, the right vertical map is injective, as claimed.

To conclude the proof, it will now suffice to prove that
the left column is exact.
To this end,  note as above that $\bigwedge\limits^tW$ has a filtration with quotients $\bigwedge\limits^iW(-Y)\otimes\bigwedge\limits^{t-i}W_Y,
i=0,...,m-m_Y$, which induces a filtration on the left column with quotients
in the form of a  tensor product of a fixed vector space with a complex
\[\bigwedge\limits^{t+1-i}W_Y\otimes H^0(E_Y(-Y.Z-L))
\to\bigwedge^{t-i}W_Y\otimes H^0(E_Y(-Y.Z))\to\bigwedge\limits^{t-i-1}W_Y
\otimes H^0(E_Y(-Y.Z+L)) .\] 
The middle cohomology of this is just $\cK_{t-i,0}(Y, E_Y(-Y.Z), L_Y, W_Y)$.
Our assumption that $t\geq m-m_Z+e_Z$ for all components $Z$ of $X$ implies that
$t-i\geq m_Y-m_W+e'_W$ for all components $W$ of $Y$ where $e'_W$ is the rank
of the restriction map $H^0(E_Y(-Z.Y))\to H^0(E_W(-Z.Y))$ (note that
via extension by zero,  $H^0(E_Y(-Z.Y))$
is contained in the image of restriction $H^0(E)\to H^0(E_Y)$). Thus,
$t-i$ satisfies a lower bound analogous to \eqref{t-lower-bound-eq}
for $(Y, E_Y(-Y.Z), L_Y, W_Y)$ in place of $(X, E, L, W)$.
Therefore by our induction on the number of components, 
the Koszul cohomology $\cK_{t-i,0}(Y, E_Y(-Y.Z), L_Y, W_Y)$ vanishes,
concluding the proof.
\qed
\subsection*{Acknowledgments} We thank the referee for helpful comments and suggestions, and F. Viviani for some
pointers to the literature.
  \bibliographystyle{amsplain}
  \bibliography{mybib}
\end{document}

%% file: myheader3.tex
\def\refp #1.{(\ref{#1})}

\newcommand{\A}{\mathcal{A}}

\def\sbr #1.{^{[#1]}}
\def\sfl #1.{^{\lfloor #1\rfloor}}

\def\?{{\bf{??}}}

\def\A{\Bbb A}

\def\O{\mathcal O}

\def\g{\mathfrak g}

\def\1/2{\frac{1}{2}}

\def\2{{[2]}}

\def\nl{\newline}

\def\<{\langle}
\def\>{\rangle}

\def\2{{[2]}}

\def\scl #1.{^{\lceil#1\rceil}}
\def\spr #1.{^{(#1)}}
\def\sbc #1.{^{\{#1\}}}

\def\subpr#1.{_{(#1)}}

\def\beq{\begin{equation*}}
\def\eeq{\end{equation*}}

\def\g3{{\Gamma\spr 3.}}

\def\cK{\mathcal{K}}

\newcommand{\eqspl}[2]{
\begin{equation}\label{#1}
\begin{split}
#2\end{split}\end{equation}}

\newcommand{\beginalphaenum}{
\begin{enumerate}\renewcommand{\labelenumi}{ }
\item \begin{enumerate}
}

\def\eex{\end{rm}\end{example}}